\documentclass[12pt]{amsart}
\usepackage{amssymb, amscd, amsmath, amsthm, amscd, curves,
epsf, epsfig, latexsym, enumerate, pstricks,graphics}

\newtheorem{theorem}{Theorem}
\newtheorem{lemma}[theorem]{Lemma}

\begin{document}
\title{3-manifolds with nilpotent embeddings in $S^4$}
 
\author{J.A.Hillman}
\address{School of Mathematics and Statistics\\
     University of Sydney, NSW 2006\\
      Australia }

\email{jonathan.hillman@sydney.edu.au}

\begin{abstract}
We consider embeddings of 3-manifolds $M$ in $S^4$ such that the two
complementary regions $X$ and $Y$ each have nilpotent fundamental group.
If $\beta=\beta_1(M)$ is odd then these groups are abelian and $\beta\leq3$.
In general $\pi_1(X)$ and $\pi_1(Y)$ have 3-generator presentations,
and $\beta\leq6$.
We give two examples illustrating our results.
\end{abstract}

\keywords{embedding, homologically balanced, nilpotent, 3-manifold, 
restrained, surgery}

\subjclass{57N13}

\maketitle

This is a continuation of the papers \cite{Hi17,Hi19},
in which we considered the complementary regions 
of a closed hypersurface in $S^4$.
Let $M$ be a closed orientable 3-manifold and
$j:M\to {S^4}=X\cup_MY$ be a locally flat embedding.
Then $\chi(X)+\chi(Y)=2$ and we may assume that $\chi(X)\leq\chi(Y)$.
In \cite[\S7]{Hi17} we considered the possibilities for $\chi(X)$, 
$\pi_X=\pi_1(X)$ and $\pi_Y=\pi_1(Y)$, and showed that if $\pi_X$ is abelian 
then $\beta=\beta_1(M;\mathbb{Q})\leq4$ or $\beta=6$,
while in \cite{Hi19} we attempted to apply 4-dimensional surgery 
to classify embeddings such that both $\pi_X$ and $\pi_Y$
are abelian.
Here we shall cast our net a little wider.
In order to use 4-dimensional surgery arguments we must restrict 
the possible groups $\pi_X$ and $\pi_Y$.
Under our present understanding of the Disc Embedding Theorem,
these groups should be in the class $G$ of groups generated from groups 
with subexponential growth by increasing unions and extensions
\cite{FT95}.
This class includes all elementary amenable groups and is included in the class of {\it restrained\/} groups, those which have no non-cyclic free subgroups.
We shall also assume that the embedding $j$ is {\it bi-epic}, i.e.,
that each of the homomorphisms $j_{X*}:\pi=\pi_1(M)\to\pi_X$ and
$j_{Y*}:\pi\to\pi_Y$ is an epimorphism.
This is so if $\pi_X$ and $\pi_Y$ are nilpotent,
since $H_1(j_X)$ and $H_1(j_Y)$ are always epimorphisms.

We show firstly that if $j$ is bi-epic and $\pi_X$ and $\pi_Y$ 
are restrained then $0\leq\chi(X)\leq\chi(Y)$, 
so $\chi(X)$ and $\chi(Y)$ are determined by $\beta$,
and if $\beta$ is even then $\chi(X)=\chi(Y)=1$ and so
$\beta_2(\pi_X;R)\leq\beta_1(\pi_X;R)$ and
$\beta_2(\pi_Y;R)\leq\beta_1(\pi_Y;R)$,
for any coefficient ring $R$.
Secondly, if $\pi_X$ and $\pi_Y$ are nilpotent then either 
$\beta=1$ or 3 and $\pi_X$ and $\pi_Y$
are free abelian groups, or $\beta=0,2,4$ or 6.
If we assume further that $\pi_X$ and $\pi_Y$ are torsion-free
nilpotent groups then there are very few known examples of such
groups with Hirsch length $>3$ and balanced presentations.
We give two examples illustrating the possibilities allowed by Theorem 2.

We shall say that an embedding has a group-theoretic property (e.g.,
abelian, nilpotent, $\dots$) 
if the groups $\pi_X$ and $\pi_Y$ have this property.

\section{restrained embeddings}

We begin with an observation that can be construed as a minimality condition.

\begin{lemma}
Let $J=j_{K,\gamma}$ be an embedding obtained from $j:M\to{S^4}$
by a proper $2$-knot surgery using the $2$-knot $K$ 
and the loop $\gamma\in\pi_{X(j)}$.
Then $J$ is not bi-epic, and $\pi_{X(J)}$ is not restrained,
unless $\pi_{X(J)}$ is itself a restrained $2$-knot group, 
in which case $\beta=\beta_1(M;\mathbb{Q})=1$ or $2$.
\end{lemma}

\begin{proof}
Let $C\cong\mathbb{Z}/q\mathbb{Z}$ be the subgroup of $\pi_{X(j)}$
generated by $\gamma$, and let $t$ be a meridian for 
the knot group $\pi{K}$.
Then 
\[
\pi_{X(J)}\cong\pi_{X(j)}*_C\pi{K}/\langle\langle{t^q}\rangle\rangle.
\]
Since the 2-knot surgery is proper,
$\langle\langle{t^q}\rangle\rangle$ is a proper normal subgroup 
of $\pi{K}$.
Since the image of $\pi_1(M)$ lies in $\pi_{X(j)}$,
the embedding $J$ cannot be bi-epic.
Moreover $\pi_{X(J)}$ can only be restrained if $\pi_{X(j)}\cong\mathbb{Z}$,
in which case $\pi_{X(J)}\cong\pi{K}$ and
$\chi(X(j))=0$ or 1, and so $\beta=1$ or 2.
\end{proof}

Let $G$ be a group.
Then $G'$ and $\zeta{G}$ shall denote the commutator subgroup 
and centre of $G$, respectively.

If $G$ is finitely generated and restrained then $def(G)\leq1$.
If $def(G)=1$ then $G$ is an ascending HNN extension, and so the first $L^2$-Betti number $\beta^{(2)}_1(G)=0$.
Hence $g.d.G\leq2$, by \cite[Theorem 2.5]{Hi}. 
The argument is homological, and so it suffices that the augmentation 
ideal in $\mathbb{Z}[G]$ have a presentation of deficiency 1 as a $\mathbb{Z}[G]$-module.
A finitely generated group $G$ is {\it balanced\/} if it has deficiency $\geq0$,
and is {\it homologically balanced\/} if
$\beta_2(G;R)\leq\beta_1(G;R)$, for any coefficient ring $R$.

Let $BS(1,m)$ be the Baumslag-Solitar group with presentation
$\langle{t,a}\mid{tat^{-1}=a^m}\rangle$, for $m\in\mathbb{Z}\setminus\{0\}$.
Then $BS(1,1)\cong\mathbb{Z}^2$, while $BS(1,-1)\cong\pi_1(Kb)$
is the Klein bottle group.
 
\begin{theorem}
Let $j:M\to{S^4}$ be a bi-epic embedding such that 
$\pi_X$ and $\pi_Y$ are restrained.

If $\beta=\beta_1(M;\mathbb{Q})$ is odd then $\chi(X)=0$ and $\chi(Y)=2$,
and $X$ is aspherical.
If, moreover,
$\pi_X$ is almost coherent or elementary amenable then
$\pi_X\cong\mathbb{Z}$ or $BS(1,m)$, for some $m\not=0$,
and $\beta=1$ or $3$.

If $\beta$ is even then $\chi(X)=\chi(Y)=1$,
and $\pi_X$ and $\pi_Y$ are homologically balanced.
\end{theorem}  

\begin{proof}
If $\pi_X$ and $\pi_Y$ are each restrained then $\chi(X), \chi(Y)\geq0$.
Hence $(\chi(X),\chi(Y))$ is determined by the parity of $\beta$,
since $0\leq\chi(X)\leq\chi(Y)\leq2$ and $\chi(X)\equiv\chi(Y)\equiv1+\beta$ 
{\it mod} (2).

Since $j$ is bi-epic, $c.d.X\leq2$ and $c.d.Y\leq2$,
by \cite[Theorem 5.1]{Hi17}.
Hence if $\chi(X)=0$ and $\pi_X$ is restrained then $X$ is aspherical,
by \cite[Theorem 2.5]{Hi}.
If, moreover, $\pi_X$ is elementary amenable or almost coherent 
then $\pi_X\cong\mathbb{Z}$ or $BS(1,m)$ for some $m\not=0$, 
by \cite[Corollary 2.6.1]{Hi}.
Hence $\beta=\beta_1(X;\mathbb{Q})+\beta_2(X;\mathbb{Q})=1$,
if $\pi_X\not\cong{BS(1,1)}=\mathbb{Z}^2$,
and $\beta=3$ if $\pi_X\cong\mathbb{Z}^2$.

Since $c.d.X\leq2$ and $X$ is homotopy equivalent to a finite complex
the cellular chain complex
$C_*(X;\mathbb{Z}[\pi_X])$ is chain homotopy equivalent to a finite free
complex $C_*$ of length 2, with $C_0$ of rank 1.
If $\beta$ is even then $\chi(X)=1$, 
and so $C_1$ and $C_2$ have the same rank.
Hence $\pi_X$ is homologically balanced.
Similarly for $\pi_Y$.
\end{proof}

The latter part of the argument shows that
the augmentation ideal in $\mathbb{Z}[\pi_X]$ 
has a square presentation matrix
(i.e., has a presentation of deficiency 0 as a $\mathbb{Z}[G]$-module).
We do not know whether the group $\pi_X$ must have deficiency $\geq0$.

There are examples of each type. (See below).
There is also a partial converse. 
If $\chi(X)=0$ and $\pi_X\cong{BS(1,m)}$ for some $m\not=0$ then 
$X$ is aspherical and
$j_{X*}$ is an epimorphism, by \cite[Theorem 5.1]{Hi17}.

\section{nilpotent embeddings}

Nilpotent embeddings are always bi-epic, 
since homomorphisms to a nilpotent group which induce epimorphisms on abelianization are epimorphisms.

\begin{theorem}
Let $j:M\to{S^4}$ be an embedding such that 
$\pi_X$ and $\pi_Y$ are nilpotent.

If $\beta=\beta_1(M;\mathbb{Q})$ is odd and $\pi_X$ and $\pi_Y$ are nilpotent then either
$X\simeq{S^1}$ and $Y\simeq{S^2}$ or
$X\simeq{T}$ and $Y\simeq{S^1\vee{S^2}\vee{S^2}}$.

If $\beta$ is even then $\beta=0,2,4$ or $6$,
and $\pi_X$ and $\pi_Y$ are each $3$-generated.
\end{theorem}

\begin{proof}

Suppose first that $\beta$ is odd.
Then $X$ is aspherical, since $\chi(X)=0$ and $\pi_X$ is nilpotent.
Hence $\pi_X\cong\mathbb{Z}$ or $\mathbb{Z}^2$, since $c.d.X\leq2$.
Since $\pi_Y$ is nilpotent and $H_1(Y;\mathbb{Z})\cong{H^2(X;\mathbb{Z})}=0$
or $\mathbb{Z}$, $\pi_Y=1$ or $\mathbb{Z}$.
The further details in this case are given in
\cite[Theorem 14]{Hi19}.

We may assume that $\beta$ is even and $\beta>4$.
Since $\chi(X)=1$ and $H_i(X;R)=0$ for $i>2$,
we have $\beta_2(X;R)=\beta_1(X;R)$,
and so $\beta_2(\pi_X;R)\leq\beta_1(\pi_X;R)$,
for any coefficient ring $R$.
Since $\pi_X$ is finitely generated and nilpotent,
there is a prime $p$ such that $\pi_X$ can be generated by $d=\beta_1(\pi_X;\mathbb{F}_p)$ elements.
Let $\widehat{\pi_X}$ be the pro-$p$ completion of $\pi_X$.
Since $\pi_X$ is nilpotent, it is $p$-good,
and so $\beta_i(\widehat{\pi_X};\mathbb{F}_p)=
\beta_i(\pi_X;\mathbb{F}_p)$, for all $i$.
The group $\widehat{\pi_X}$ is a pro-$p$ analytic group, 
and so has a minimal presentation with 
$d=\beta_1(\widehat{\pi_X};\mathbb{F}_p)$ generators and
$r=\beta_2(\widehat{\pi_X};\mathbb{F}_p)$ relators.
Since $\beta>2$, $\widehat{\pi_X}\not\cong\widehat{\mathbb{Z}}_p$,
and so $r>\frac{d^2}4$, by \cite[Theorem 2.7]{Lu83}.
(Similarly for $\pi_Y$.)
Therefore $d\leq3$ and $\beta\leq2d\leq6$.
\end{proof}

If $\pi_X$ is nilpotent and $\beta_1(X;\mathbb{Q})=0$ then $\pi_X$ is finite,
while if $\beta_1(X;\mathbb{Q})=1$ then $\pi_X\cong{F\rtimes\mathbb{Z}}$,
where $F$ is finite.
Thus if $\pi_X$ and $\pi_Y$ are torsion-free nilpotent and $\beta\leq3$
then $\pi_X$ and $\pi_Y$ are abelian.
See \cite[Theorems 10, 11 and 16]{Hi19} for more on such embeddings.

It is reasonable to restrict consideration further to torsion-free nilpotent groups, as such groups satisfy the Novikov conjecture, and the surgery obstructions are maniable.

If $G$ is torsion-free nilpotent of Hirsch length $h$ then $c.d.G=h$.
The first non-abelian examples are the $\mathbb{N}il^3$-groups $\Gamma_q$,
with presentations $\langle{x,y,z}\mid[x,y]=z^q,~[x,z]=[y,z]=1\rangle$.
Some of the argument of \cite[Theorem 18]{Hi19} for the group $\mathbb{Z}^3$
extends to the groups $\Gamma_q$.
The homology of the pair $(X,M)$ with coefficients $\mathbb{Z}[\pi_X]$
gives an exact sequence
\[
H_2(X;\mathbb{Z}[\pi_X])\to{H^2(X;\mathbb{Z}[\pi_X])}\to
{H_1(M;\mathbb{Z}[\pi_X])}\to0.
\]
Let $K_X=\mathrm{Ker}(j_{X*})={H_1(M;\mathbb{Z}[\pi_X])}$ 
and $P=H_2(X;\mathbb{Z}[\pi_X])$.
Since $c.d.X\leq2$ and $c.d.\Gamma_q=3$ we see that 
$P$ is a projective $\mathbb{Z}[\Gamma_q]$-module of rank 1.
It is stably free since $\widetilde{K}_0(\mathbb{Z}[G])=0$ for
torsion-free poly-Z groups $G$, 
and $P$ has rank 1 since $\chi(X)=1$.
Since $Ext^i_{\mathbb{Z}[\pi_X])}(\mathbb{Z},\mathbb{Z}[\pi_X])=0$ 
for $i\leq2$
we then see that $H^2(X;\mathbb{Z}[\pi_X])\cong{P^\dagger}=
\overline{Hom_{\mathbb{Z}[\pi_X])}(P,\mathbb{Z}[\pi_X])}$,
and so is also stably free of rank 1.
We thus have an exact sequence 
\[
P\to{P^\dagger}\to{K_X}\to0.
\]
However it is not clear that this is as potentially useful as the analogous conditions on abelian embeddings in Theorems 10, 12, 14 and 18 of \cite{Hi19}.
Moreover, if $G$ is a nonabelian poly-Z group 
then there are infinitely many isomorphism classes of 
stably free $\mathbb{Z}[G]$-modules $P$ such that
${P\oplus\mathbb{Z}[G]}\cong\mathbb{Z}[G]^2$ \cite{Ar81}.
We do not know which can be realized as $H_2(X;\mathbb{Z}[\pi_X])$,
for an embedding $j$ with $\pi_X\cong\Gamma_q$.
(This contrasts strongly with the case $\pi_X\cong\mathbb{Z}^3$,
for then $P$ is a free module.)

It can be shown that there is just one 
homologically balanced torsion free nilpotent group with Hirsch
length $h=4$, and none with $h=5$ or with $h=6$ and $\beta=3$
\cite{Hi20}.
The example with $h=4$ is an extension of $\mathbb{Z}^2$ by
$\mathbb{Z}^2$ with presentation 
\[
\langle{t,u}\mid[t,[t,[t,u]]]=[u,[t,u]]=1\rangle.
\]
If we consider more general solvable groups, 
we can find many metabelian groups of Hirsch length 5 with balanced presentations.
One such group has presentation 
\[
\langle{t,x}\mid{t^4xt^{-4}=t^2x^2t^{-1}x^{-1}t^{-1}x^{-1}},
~xt^2xt^{-2}=t^2xt^{-2}x\rangle.
\]
This example is the group of a Cappell-Shaneson 3-knot, 
with commutator subgroup $\mathbb{Z}^4$. 
If $G$ is torsion-free nilpotent and $h(G)>>5$ is $def(G)<0$?
In particular, is this so if $h(G)\geq6$ and $G/G'\cong\mathbb{Z}^3$?

\section{examples}

Pairs of groups with balanced presentations and isomorphic abelianizations 
can be realized by embeddings of 3-manifolds \cite{Li04}.

Our examples are based on 3- and 4-component links
$L=L_a\cup{L_u}\cup{L_v}$ or $L_a\cup{L_b}\cup{L_u}\cup{L_v}$,
where $L_a\cup{L_b}$ and $L_u\cup{L_v}$ are trivial links.
The 3-manifold $M$ obtained by 0-framed surgery on $L$ embeds in $S^4$,
and the complementary regions have Kirby-calculus presentations
in which one of these sublinks is 0-framed and the other dotted
(the roles being exchanged for the two regions).
The components $L_a$ and $L_b$ represent words $A$ and $B$ in $F(u,v)$ and
the components $L_u$ and $L_v$ represent words $U$ and $V$ in $F(a,b)$,
and we may arrange that $\pi_X$ and $\pi_Y$ have presentations
$\langle{u,v}\mid{A,B}\rangle$ and $\langle{a,b}\mid{U,V}\rangle$, respectively.
(See the Figure.)
The embeddings constructed in this way are always bi-epic, 
since $\pi_X$ and $\pi_Y$ are generated by the images of meridians of $L$.
We shall use the tabulation of links in \cite{Rol}.

\setlength{\unitlength}{1mm}
\begin{picture}(90,45)(-28,10)

\put(-2,19.1){$\vartriangleright$}
\put(-4,22){$u$}
\put(20,20){$0$}
\put(-16,34){$\bullet$}
\put(78,34){$\bullet$}
\put(63,19.1){$\vartriangleright$}
\put(61,22){$v$}
\put(29,48.5){$\vartriangleright$}
\put(27,50){$a$}

\curve(11,43.5,0,50, -10.6,45.6,-15,35,-10.6,24.6,0,20,10.5,24)

\qbezier(12,42)(15,35)(11.6,26)

\qbezier(9.5,47)(30,52)(54,46.5)
\qbezier(10.8,45)(30,50)(51,44.5)
\qbezier(54,46.5)(57,45)(53.5,44)
\qbezier(9,44.4)(7.5,42.9)(10,42.5)

\qbezier(10,42.5)(30,46)(44.2,43)
\qbezier(11,25)(30,22)(44.2,25)

\qbezier(4.5,30)(5,26)(11,25)

\qbezier(7.5,46)(5,45)(4.55,42)
\qbezier(4.5,42)(2,36)(4.5,30)

\curve(55,47,64,50, 74.6,45.6,79,35,74.6,24.6,64,20,55,23)

\qbezier(53.6,45.7)(52,44.4)(50.8,39)
\qbezier(55,23)(52,25)(50.8,29)

\qbezier(47,39)(49,39)(52,39)
\qbezier(47,29)(49,29)(52,29)
\qbezier(47,39)(47,34)(47,29)
\qbezier(52,39)(52,34)(52,29)

\qbezier(48.2,39)(48.2,42)(44.2,43)
\qbezier(48.2,29)(48.2,26)(44.2,25)

\put(47.8,33.4){$m$}

\put(25, 12){Figure 1}
\put(17,34){$A=uvu^{-1}v^{-m}$}

\end{picture}

For example, consider the 3-component link in Figure 1, 
in which the strands in the box have $m$ full twists,
$A=uvu^{-1}v^{-m}$, $B=U=1$ and $V=a^{m-1}$.

When $m=0$ the link is the split union of an unknot and the Hopf link $2^2_1$,
$M\cong{S^2}\times{S^1}$,
$X\cong{D^3}\times{S^1}$ and $Y\cong{S^2}\times{D^2}$.
When $m=1$ the link is the Borromean rings $6^3_2$,
and $X$ is a regular neighbourhood of the unknotted embedding of the torus
$T$ in $S^4$.
When $m=-1$, the link is $8^3_9$,
and $X$ is a regular neighbourhood of the unknotted embeding of the 
Klein bottle $Kb$ in $S^4$ with normal Euler number 0.
In general, $X$ is aspherical, $\pi_X\cong{BS(1,m)}$ and $\pi_Y\cong\mathbb{Z}/(m-1)\mathbb{Z}$.
(Note however that the boundary of a regular neighbourhood of the Fox 2-knot 
with group $BS(1,2)$ gives an embedding of $S^2\times{S^1}$ with
$\pi_X\cong{BS(1,2)}$ and $\chi(X)=0$, 
but this embedding is not bi-epic and $X$ is not aspherical.)

We may also construct embeddings such that $\pi_X\cong{BS(1,m)}$
and $\chi(X)=1$, while $\pi_Y\cong {BS(1,m)}$
or $\mathbb{Z}\oplus\mathbb{Z}/(m-1)\mathbb{Z}$.
These require 4-component links.

This is also the case if $\pi_X\cong\Gamma_q$,
for then $\beta_2(X)=\beta_1(X)=2$.
If $\pi_Y$ is abelian then $\pi_Y\cong\mathbb{Z}^2$
\cite[Theorem 7.1]{Hi17}, and so $q=1$.

It is easy to find a 4-component link $L=L_a\cup{L_b}\cup{L_u}\cup{L_v}$ 
with each 2-component sublink trivial, 
and such that $L_a$ and $L_b$ represent 
(the conjugacy classes of) $A=[u,[u,v]]$ and $B=[v,[u,v]]$ in $F(u,v)$,
respectively, while $L_u$ and $L_v$ have image 1 in $F(a,b)$.
Arrange the link diagram so that $L_u$ is on the left, 
$L_v$ on the right, $L_a$ at the top and $L_b$ at the bottom. 
We may pass one bight of $L_a$ which loops around $L_u$ under a 
similar bight of $L_b$, so that $U$ now represents $[a,b]$ in $F(a,b)$.
Finally we use claspers to modify $L_u$ and $L_v$ so 
that they represent $[b,v]$ in $F(b,v)$ and $[a,u]$ in $F(a,u)$.
We obtain the link of Figure 2.

\setlength{\unitlength}{1mm}
\begin{picture}(90,85)(-32,-5)

\put(-10,9.05){$\vartriangleright$}
\put(-12.5,11.3){$u$}
\put(59,8){$v$}
\put(61,6.05){$\vartriangleright$}
\put(39,56){$\vartriangleright$}
\put(36.5,58.3){$a$}
\put(16,15){$\vartriangleright$}
\put(13.5,12.5){$b$}

\linethickness{1pt}
\put(-20,71){\line(1,0){20}}
\put(-20,10){\line(0,1){61}}
\put(-20,10){\line(1,0){73}}
\put(-6,68){\line(1,0){76}}
\put(-6,62){\line(1,0){5}}
\put(-6,62){\line(0,1){6}}
\put(-2,66){\line(1,0){52}}
\put(-2,64){\line(1,0){1}}
\put(-2,64){\line(0,1){2}}
\put(1,64){\line(1,0){9}}
\put(1,62){\line(1,0){7}}
\put(0,69){\line(0,1){2}}
\put(0,61){\line(0,1){4}}
\put(8,61){\line(0,1){1}}
\put(10,58){\line(0,1){6}}
\put(8,58){\line(1,0){2}}
\put(8,58){\line(0,1){1}}
\put(50,7){\line(1,0){20}}
\put(70,7){\line(0,1){61}}
\put(50,58){\line(0,1){8}}
\put(0,53){\line(0,1){6}}
\put(0,49){\line(0,1){2}}
\put(50,51){\line(0,1){5}}

\put(50,48.5){\line(0,1){1}}

\put(0,46.5){\line(0,1){1}}
\put(0,40){\line(0,1){5}}
\put(0,31){\line(0,1){7}}
\put(0,28.5){\line(0,1){1}}
\put(50,33){\line(0,1){14}}
\put(50,26.5){\line(0,1){1}}
\put(50,29){\line(0,1){2}}
\put(40,12){\line(1,0){9}}
\put(51,12){\line(1,0){2}}
\put(53,10){\line(0,1){2}}
\put(40,12){\line(0,1){3}}
\put(40,17){\line(0,1){1}}
\put(40,18){\line(1,0){2}}
\put(42,14){\line(0,1){4}}
\put(42,14){\line(1,0){7}}
\put(51,14){\line(1,0){7}}
\put(58,14){\line(0,1){6}}
\put(0,20){\line(1,0){58}}
\put(0,20){\line(0,1){7}}

\put(50,21){\line(0,1){4}}
\put(50,11){\line(0,1){4}}
\put(50,17){\line(0,1){2}}
\put(50,7){\line(0,1){2}}

\thinlines
\put(-2,60){\line(1,0){11}}
\put(1,57){\line(1,0){51}}
\put(10,60){\line(1,0){20}}
\put(30,58){\line(0,1){2}}
\put(-2,57){\line(1,0){1}}
\put(-2,57){\line(0,1){3}}
\put(1,57){\line(1,0){51}}
\put(1,54){\line(1,0){48}}
\put(-2,54){\line(1,0){1}}
\put(-2,52){\line(1,0){51}}
\put(-2,52){\line(0,1){2}}
\put(-2,50){\line(1,0){1}}
\put(1,50){\line(1,0){51}}
\put(-2,48){\line(1,0){54}}
\put(-2,48){\line(0,1){2}}
\put(52,50){\line(0,1){2}}
\put(52,54){\line(0,1){3}}
\put(51,52){\line(1,0){1}}
\put(51,54){\line(1,0){1}}
\put(-2,46){\line(1,0){12}}
\put(-2,44){\line(1,0){1}}
\put(-2,44){\line(0,1){2}}
\put(1,44){\line(1,0){7}}
\put(-2,41){\line(1,0){1}}
\put(1,41){\line(1,0){29}}
\put(-2,39){\line(1,0){51}}
\put(51,39){\line(1,0){1}}
\put(52,39){\line(0,1){9}}
\put(-2,39){\line(0,1){2}}
\put(30,41){\line(0,1){6}}
\put(10,42){\line(0,1){4}}
\put(8,42){\line(0,1){2}}
\put(10,36){\line(1,0){39}}
\put(8,34){\line(1,0){41}}
\put(51,36){\line(1,0){3}}
\put(51,34){\line(1,0){1}}
\put(10,36){\line(0,1){2}}
\put(8,34){\line(0,1){4}}
\put(1,32){\line(1,0){51}}
\put(52,32){\line(0,1){2}}
\put(-2,32){\line(1,0){1}}
\put(-2,30){\line(1,0){51}}
\put(51,30){\line(1,0){1}}
\put(-2,30){\line(0,1){2}}

\put(-2,28){\line(1,0){54}}
\put(-2,26){\line(1,0){1}}
\put(-2,26){\line(0,1){2}}
\put(1,26){\line(1,0){51}}
\put(52,24){\line(0,1){2}}
\put(-2,24){\line(1,0){1}}
\put(1,24){\line(1,0){48}}
\put(51,24){\line(1,0){1}}
\put(52,28){\line(0,1){2}}

\put(-2,16){\line(1,0){43}}
\put(-2,16){\line(0,1){8}}
\put(45,16){\line(1,0){9}}
\put(54,16){\line(0,1){3}}
\put(54,21){\line(0,1){15}}

\put(20,0){Figure 2}

\end{picture}

This link may be partitioned into two trivial links in three distinct ways,
giving three embeddings of the 3-manifold obtained by 0-framed surgery on $L$.
If the two sublinks are $L_a\cup{L_b}$ and $L_u\cup{L_v}$ 
then 
\[
A=vu^{-1}v^{-1}u^{-1}vuv^{-1}u,~B=vuv^{-1}u^{-1}v^{-1}uvu^{-1},
\]
\[
U=b^{-1}aba^{-1}\quad\mathrm{and}\quad{V=1}.
\] 
Hence $\pi_X\cong\Gamma_1$ and $\pi_Y\cong\mathbb{Z}^2$.

Each of the other partitions determine abelian embeddings, with
$\pi_X\cong\pi_Y\cong\mathbb{Z}^2$ and $\chi(X)=\chi(Y)=1$.

With a little more effort, instead of passing just one bight of $L_a$ 
under $L_b$ (as above), we may interlace the loops of $L_a$ and $L_b$ 
around each of $L_u$ and $L_v$ so that $u$ and $V$ represent $[a,[a,b]]$ 
and $[b,[a,b]]$, respectively, 
and so that each 2-component sublink of $L$ is still trivial.
If we then use claspers again we may arrange that $u$ 
represents $[a,v]$ and $v$ represents $[b,u]$,
so that we obtain a 3-manifold which has one embedding with $\pi_X\cong\pi_Y\cong\Gamma_1$ and another with
$\pi_X\cong\pi_Y\cong\mathbb{Z}^2$.
Can we refine this construction so that the third embedding has
$\pi_X\cong\Gamma_1$ and $\pi_Y\cong\mathbb{Z}^2$?

\smallskip
{\it Acknowledgment.} This work was begun at the BIRS conference on 
``Unifying Knot Theory in Dimension 4", 4-8 November 2019.

%\newpage 

\end{document}